\documentclass{amsart}
\usepackage{amssymb}
\usepackage{latexsym}

\newtheorem{theorem}{Theorem}
\newtheorem{lemma}[theorem]{Lemma}
\newtheorem{corollary}[theorem]{Corollary}

\newtheorem{example}[theorem]{Example}

\newcommand{\uu}{\mathbf{u}}
\newcommand{\vv}{\mathbf{v}}

\newcommand{\ww}{\mathbf{w}}

\def\x{{\bf x}}
\def\y{{\bf y}}

\newcommand{\Z}{\mbox{${\mathbb
 Z}$}}

\title{Counting patterns in colored orthogonal arrays}

\author{A. Montejano and O. Serra
}

\address{Universitat Polit\`ecnica de Catalunya\\ Jordi Girona, 1,
E-08034 Barcelona, Spain\\e-mail:oserra@mat.upc.es}

\begin{document}

\maketitle

\begin{abstract} Let $S$ be an orthogonal array $OA(d,k)$  and let $c$ be an
$r$--coloring of its ground set $X$. We give a combinatorial
identity which relates the number of vectors in  $S$ with given
color patterns under  $c$ with the cardinalities of the color
classes.  Several applications of the identity are considered. Among
them, we show that every equitable $r$--coloring of the integer
interval $[1,n]$ has at least $\frac{1}{2}(\frac{n}{r})^2+O(n)$
monochromatic Schur triples. We also show that in an orthogonal array $OA(d,d-1)$,
the number of monochromatic vectors of each color depends only on the number of vectors which
miss that color and the cardinality of the color class.
\end{abstract}

\section{Introduction}

Arithmetic Ramsey Theory can be seen as the study of   the existence
of monochromatic structures, like arithmetic progressions or
solutions of linear systems, in every coloring of sets of integers.
The early results in the area are the theorem of Schur on
monochromatic solutions of the equation $x+y=z$, the Van der Waerden
theorem on monochromatic arithmetic progressions or the common
generalization of the two, Rado's theorem, on monochromatic
solutions of linear systems. Anti--Ramsey results refer to the study
of combinatorial structures with elements of pairwise distinct
colors, or rainbow structures, a subject started by Erd\H{o}s,
Simonovits and S\'os \cite{ess} which has received much attention
since then. Canonical Ramsey theory collects results ensuring the
existence of either a monochromatic or a rainbow structure.
Jungi\'c,  Licht,  Mahdian, Ne\v{s}et\v{r}il and Radoi\v{c}i\'c
\cite{jlmnr} started what they call Rainbow Ramsey Theory, which
concerns the study of rainbow structures in colourings of sets of
integers.

Counting versions of Arithmetic Ramsey results have also been
obtained. Frankl, Graham and R\"odl \cite{fgr88} prove that, in a
finite coloring of an integer interval, actually a positive fraction
of all solutions of a partition regular system are monochromatic.
They also prove that, for every coloring of an integer interval, a
positive fraction of all solutions of such a system are either
monochromatic or rainbow.

Some of the above phenomena on the existence and number  of color
patterns in combinatorial structures behave in a particularly nice
way when considered in finite groups. A simple example is the fact
that the total number of monochromatic Schur triples in every
two--coloring of the group of integers modulo $n$ depends only on
the cardinality of the color classes but not on the distribution of
the colors, a fact first noticed, as far as we know,  by Datskowsky
\cite{d}. The same is true for monochromatic three--term arithmetic
progressions when $n$ is relatively prime with $6$, as noted by
Croot \cite{croot}. In \cite{ccs} a combinatorial counting argument
was given which explains the above two results and provides the
ground for further generalizations in three directions. First,
results like the above mentioned ones can be extended to general
finite groups. Actually the universe to be colored needs not to be
even a group, but simply the base set of an orthogonal array.
Second, the monochromatic structures include Schur triples,
arithmetic progressions, or solutions of more general equations in
groups. Third, the counting argument can be applied to colorings
with more than two colors and can also be used  to study rainbow
structures or specific color patterns. Of course there are
limitations in such general results, which become less precise with
the increasing complexity of the structures we consider.

We give in Section \ref{sec:oa} a general formulation of the basic
counting lemma (Lemma \ref{lem:counting}) which is based in a
counting argument used in \cite{ccs}. Section \ref{sec:OA(3,2)}
collects some specific applications for orthogonal arrays $OA(3,2)$,
which include solutions of linear equations $ax+by+cz=d$ in an
abelian group of order coprime with $a$, $b$ and $c$ or, more
generally, equations of the form $x^{\alpha}y^{\beta}z^{\gamma}=b$
in a group $G$, where $\alpha, \beta,\gamma$ are automorphisms of
$G$. In this case Lemma \ref{lem:counting} leads to a relationship
between monochromatic and rainbow triples which depends only on the
cardinalities of the color classes (Theorem \ref{thm:2m-r}). This
relationship provides results on the minimum number of monochromatic
or of rainbow triples in orthogonal arrays (Corollary
\ref{cor:density} and Corollary \ref{cor:(3,2)-3}).

In the general context of orthogonal arrays $OA(3,2)$ the
relationship between monochromatic and rainbow triples can not be
strengthened as illustrated in Example \ref{exm:feasible}. Section
\ref{sec:linear} particularizes to linear equations of the form
$ax+by+cz=d$ in an abelian group. In particular Theorem
\label{thm:2m-r} is used to obtain a lower bound on the number of
monochromatic Schur triples in an equitable $r$--coloring of the
integer interval $[1,n]$ (Theorem \ref{thm:schur}). For $3$--term
arithmetic progressions, colorings which are rainbow--free have been
characterized in \cite{ms09}. This characterization shows that every
coloring with smaller color class sufficiently large has a rainbow
triple. Here we obtain a similar result for general groups non
necessarily abelian (Corollary \ref{cor:3ap}).  It is also shown
that every equitable coloring of an abelian group has at least a
linear number of rainbow solutions of any given linear equation
which is partition regular (Corollary \ref{cor:rainbowpr}).

Section \ref{sec:d,d-1} is devoted to orthogonal arrays of the form
$OA(d,d-1)$. The main result establishes a relationship between
monochromatic vectors of a given color and vectors which miss that
color, which again depends only on the cardinality of the color
class (Theorem \ref{thm:d,d-1}). Particular instances are the main
result in \cite{ccs} on monochromatic vectors in $2$--colorings, and
a result by Balandraud \cite{balandraud} on rainbow vectors in
$3$--colorings.

The paper closes by considering the general case of orthogonal
arrays $OA(d,k)$ in Section \ref{sec:d,k} where Lemma
\ref{lem:counting} is used to show that, for every $r$--coloring of
the base set of an orthogonal array $OA(d,k)$,  a positive
proportion of all vectors has patterns in every ball of radius
$(d-k)$, where we use the $\ell_1$ distance in the set of
$r$--vectors identifying color patterns (Theorem \ref{thm:asym}).
This result provides a quantitative estimation which is
particularized in the case of almost monochromatic (all but at most
one entry of the same color) or almost rainbow (all but at most one
of the entries pairwise distinct) patterns (Theorem
\ref{thm:monorain}).

\section{A counting argument}\label{sec:oa}

Let $X$ be a finite set with cardinality $n$ and let $S$ be a set of
vectors in $X^d$. Let $c:X\rightarrow [1,r]$ be an $r$--coloring of
$X$ with color classes $X_1, \ldots ,X_r$.
 A vector  $\x=(x_1, x_2,\ldots, x_d)\in S$ is {\it
monochromatic} under $c$ if all its coordinates belong to the same
color class. When there are either no two coordinates of the same
color class or all colors are present, we say that the vector is
{\it rainbow} under $c$. We denote by $M=M(S)$ and $R=R(S)$ the set
of monochromatic and rainbow vectors in $S$ respectively.

A set $S$ of $d$-vectors with entries in
  $X$ is an \emph{orthogonal array} of degree~$d$ and
strength~$k$  if, for any choice of $k$ columns, each $k$-vector of
$X^k$ appears in exactly one vector of $S$. In other words, if we
specify any set of $k$ entries   $a_1,\cdots ,a_k$ and any set of
subscripts $1\le i_1<i_2<\cdots<i_k\le d$,  we find exactly one
vector $\y =(y_1, y_2,\ldots ,y_d)$ in $S$ with $y_{i_1}=a_1,
y_{i_2}=a_2,\ldots ,y_{i_k}=a_k$. We denote by $OA(d,k)$ the family
of orthogonal arrays of degree $d$ and strength $k$ on $X$.

Lemma \ref{counting} below is the basic tool we shall use. It is
based on the counting arguments used in \cite{ccs}.

In what follows we use the following notation. The color classes of
an $r$--coloring   of $X$ will be denoted by $X_1, X_2,\ldots ,X_r$,
and we denote by  $c_i=|X_i|/n$ the density of the $i$-th color
class. For a vector $\uu=(u_1,\ldots ,u_r)$ with nonnegative integer
entries, we denote by $| \uu |=\sum_i u_i$. The multinomial
coefficient ${d\choose u_1,u_2,\ldots ,u_r,d-| \uu|}$ will be
written as ${d \choose \uu}$. For a vector $\vv=(v_1,\ldots ,v_r)$
we write ${\vv\choose \uu}={v_1\choose u_1}{v_2\choose u_2}\cdots
{v_r\choose u_r}$. We use the convention ${v\choose u}=0$ if $v<u$
and ${0\choose 0}=1$.

\begin{lemma}\label{lem:counting}
Let $S$ be an orthogonal array  $OA(d,k)$ on  $X$ and let $c$ be an
$r$--coloring of $X$.

For each vector $\uu=(u_1,u_2,\ldots ,u_{r})$ with $|\uu|\le k$ the
following equality holds:
\begin{equation}\label{eq:counting}
\frac{1}{n^{k}}\sum_{|\vv|=d} {\vv \choose \uu}s(\vv)= {d\choose
\uu}c_1^{u_1}\cdots c_r^{u_r},
\end{equation}
where the sum is extended to all vectors $\vv= (v_1, v_2,\ldots
,v_r)$ with nonnegative integer entries, $|\vv|=d$, and $s(\vv)$ is
the number of vectors in $S$ with   $v_i$ coordinates in $X_i$ for
each $i=1,\ldots ,r$.
\end{lemma}

\proof Given an ordered partition $V=(V_1,V_2, \ldots ,V_r)$ of
$[1,d]$, possibly with some empty parts, let us denote by $S(V)$ the
set of vectors in $S$ whose entries in $V_i$ belong to $X_i$, $1\le
i \le r$. An $r$-tuple of subsets $(U_1,U_2,\ldots ,U_r)$ of $[1,d]$
is of type $\uu=(u_1,u_2,\ldots ,u_r)$ if $|U_i|=u_i$, $1\le i\le
r$. Denote by ${\mathcal P}^r(\uu)$ the set of all $r$--tuples of
pairwise disjoint subsets of $[1,d]$ of type $\uu$. We say that $V$
dominates $U$, and write $V\succeq U$, if $V_i\supset U_i$, $1\le
i\le r$.

Since $S$ is an orthogonal array $OA(d,k)$, there are $k-| \uu|$
vectors in $S$ which meet a prescribed assignment of $| \uu|$
coordinates. Hence, for each  $r$--tuple of subsets $(U_1,U_2,\ldots
U_r)$ in ${\mathcal P}^r(\uu)$   there are
 $|X_1|^{u_1}|X_2|^{u_2}\cdots
|X_r|^{u_r}|X|^{k-| \uu|}$ vectors in $S$ whose entries in $U_i$
belong to $X_i$, $1\le i \le r$. Among these vectors we find   all
vectors in $S(V)$ for each partition $V$ which dominates $U$, that
is,
$$
\sum_{V\succeq U}S(V)=|X_1|^{u_1}|X_2|^{u_2}\cdots
|X_r|^{u_r}|X|^{k-| \uu|}.
$$
Each partition $V$ dominates ${|V_1|\choose u_1}{|V_2|\choose
u_2}\cdots {|V_r|\choose u_r}$   $r$--tuples  in ${\mathcal
P}^r(\uu)$. Summing up through all $r$--tuples in ${\mathcal
P}^r(\uu)$ we get vectors counted by $s(\vv)$ for each $\vv$ which
dominates componentwise the vector $\uu$:
\begin{eqnarray*}
{d\choose \uu}|X_1|^{u_1}|X_2|^{u_2}\cdots |X_r|^{u_r}|X|^{k-|
\uu|}&=&\sum_{U\in {\mathcal
P}^r(\uu)}\sum_{V\succeq U}S(V)\\
&=&\sum_{V\succeq U}\sum_{U\in {\mathcal P}^r(\uu)}S(V)\\
&=&\sum_{|\vv|=d} {v_1\choose u_1}{v_2\choose u_2}\cdots {v_r\choose
u_r}\sum_{V\in {\mathcal P}^r(\vv)}S(V)\\&=&\sum_{|\vv|=d}
{v_1\choose u_1}{v_2\choose u_2}\cdots {v_r\choose u_r}s(\vv).
\end{eqnarray*}
Dividing by $n^{k}$ we get  equation \eqref{eq:counting}.\qed

\medskip

Lemma \ref{lem:counting} gives a relationship between the number of
vectors with some specific color patterns and the cardinalities of
the color classes. This identity may provide some precise formulas
for the number of vectors with a particular color pattern, or at
least approximate   counting results of a   general nature. In the
remaining of the paper we give some applications of these
identities.

\section{Colour patterns in $OA(3,2)$}\label{sec:OA(3,2)}

 For orthogonal arrays $O(3,2)$ we get a
nice relationship between monochromatic and rainbow vectors.

\begin{theorem}\label{thm:2m-r} Let  $S$ be an orthogonal array
$OA(3,2)$ on $X$ and $n=|X|$. For any  $r$--coloring of $X$ we have
\begin{equation}\label{eq:2m-r}
2|M|-|R|=n^2(3\sum_{i=1}^r c_i^2-1),
\end{equation}
where $M$ and $R$ denote the set of monochromatic and   rainbow
vectors of $S$ respectively.
\end{theorem}

\begin{proof} By taking $\uu=(0,0,0)$ in Lemma \ref{lem:counting} we get
\begin{equation}\label{eq:1}
|X|^2=\sum_{\|v\|=3} s(\vv)=|M|+|R|+|T(2,1)|,
\end{equation}
where $T(2,1)=S\setminus \{ M\cup R\}$ denotes the set of vectors in
$S$ with exactly two entries of the same colour.

On the other hand, the choice of $\uu=(2,0,0)$ in Lemma
\ref{lem:counting} gives
$$
3|X_1|^2=3 s(3,0,0)+s(2,1,0)+s(2,0,1). $$ Adding up similar
countings with $(0,2,0)$ and $(0,0,2)$, we have
\begin{equation}\label{eq:2}
3\sum_{i=1}^r |X_i|^2=3|M|+|T(2,1)|.
\end{equation}
The result follows by substracting (\ref{eq:1}) from (\ref{eq:2}).
\end{proof}

As an  immediate consequence of Theorem \ref{thm:2m-r} we get:

\begin{corollary}\label{cor:density} Let $c$ be an $r$--coloring  of the base set of an
orthogonal array $OA(3,2)$ with $\alpha_c=3\sum_{i=1}^r c_i^2-1.$ If
$\alpha_c>0$ then  there are at least  $\alpha_c n^2$
 monochromatic triples and, if $\alpha_c<0$, then there are at least $|\alpha_c|n^2$ rainbow triples.

 In particular, every equitable coloring with $r\ge 4$ colors has at
 least $$(1-3/r)n^2$$ rainbow triples.
\end{corollary}

In the context of orthogonal arrays there are examples which show
that essentially all solutions for $|M|$ and $|R|$ in equation
\eqref{eq:2m-r} are possible values for the number of monochromatic
and rainbow vectors in an orthogonal array whose points are colored.
We illustrate this fact with the following example.

\begin{example}\label{exm:feasible} {\rm Let $Y$ be a multiplicative quasigroup (we only require the cancellation  law)
and consider the quasigroup $X=Y\times \Z_3$ with
$(x,i)*(y,j)=(xy,2(i+j))$. The set of triples $\{(x,i), (y,j),
(x,i)*(y,j): (x,i), (y,j)\in X\}$ is an orthogonal array. The
coloring $\chi (x,i)=i$ on $X$ has the maximum possible number
$3|Y|^2$ of monochromatic triples and the maximum possible number
$6|Y|^2$ of rainbow triples for an equitable coloring of $X$.

Let $L$ be the latin square on $X$ with entries
$L((x,i),(y,j))=(x,i)*(y,j)$ for each $(x,i), (y,j)\in X$.  Let
$U,V\subset Y$ be subsets of $Y$. Exchange the entries in $L$ of the
form $(xy,0)$ with $(xy,1)$ for every $x\in U$ and $y\in V$. The
resulting orthogonal array has $3|Y|^2-2|S|\cdot|T|$ monochromatic
triples for the same coloring of $X$.

Let $L'$ be the latin square  obtained by the above procedure with
$U=V=Y$. For each pair $U',V'\subset Y$ we can now exhange the
entries $(xy,0)$ with $(xy,2)$ whenever $x\in U'$ and $y\in V'$. The
resulting orthogonal array has $|Y|^2-|U'|\cdot |V'|$ monochromatic
triples for the same coloring. By choosing $U'=V'=Y$, there are no
monochromatic, and therefore no rainbow, triples. These are examples
of equitable colorings in orthogonal arrays for each value of
$|M|\in [0,|Y|^2]\cup (|Y|^2+2\cdot [0,|Y|^2])$.}\qed
\end{example}

For two--colorings there are no rainbow triples, so that Theorem
\ref{thm:2m-r} gives a formula for the total number of monochromatic
triples in terms of the cardinalities of the color classes. By
minimizing that formula (with each color class of density $1/2$) we
get the minimum number of monochromatic triples in an orthogonal
array $OA(3,2)$ for any two-coloring of its ground set $X$. More
precisely, we have the next Corollary which is a natural
generalization of Corollary 3.1 in \cite{ccs},

\begin{corollary}\label{cor:(3,2)-2}
Let  $S$ be an orthogonal array $OA(3,2)$ on $X$. For any 2-coloring
of $X$ we have
$$
|M|=|X_1|^2-|X_1|\cdot|X_2|+|X_2|^2. $$ In particular, for any
$2$--coloring of $X$, there are at least $n^2/4$ monochromatic
triples in $S$ .
\end{corollary}

In the case of three--colorings Theorem \ref{thm:2m-r} has a nice
interpretation. Let us call $$\sigma^2_c=\sum_{i=1}^r
c_i^2/r-(\sum_{i=1}^r c_i/r)^2$$ the {\it variance} of an
$r$--coloring $c$. For $r=3$ the expression on the right of equation
\eqref{eq:2m-r} coincides, up to a constant, with the variance of
the coloring (in particular is always nonnegative.) Theorem
\ref{thm:2m-r} can be restated for three--colorings  in the
following form.

\begin{corollary}\label{cor:(3,2)-3}
Let  $S$ be an orthogonal array $OA(3,2)$ on $X$. For any 3-coloring
of $X$ we have
\begin{equation}\label{eq:var}
2|M|-|R|=9\sigma^2_c n^2 .
\end{equation}
In particular,  there are at least $(9\sigma^2_c/2)n^2$
monochromatic triples.
\end{corollary}

\section{Linear equations}\label{sec:linear}

Natural extensions of results in Arithmetic Ramsey Theory concern
the study of color patterns of structures in groups. The results in
Section \ref{sec:oa} can be directly applied to this setting. The
set of solutions of a linear equation of the form $$ ax+by+cz=d$$ in
an abelian group of order coprime with $a,b$ and $c$ forms an
orthogonal array $OA(3,2)$. In this case more precise information on
the number of monochromatic ir rainbow triples can be obtained,
usually depending on the particular equation we are considering.

There are  colorings with only rainbow triples. Take for instance
the set of solutions of the equation $$x+y+z=-1$$ in a cyclic group
of order $n\equiv 0 \pmod{3t}$ for some $t\ge 1$. Consider the
partition $A_i=[0, (n/3t)-1]+i(n/3t)$, $0\le i\le 3t-1$. We have
$A_i+A_i=[0,2(n/3t)-2]+2i (n/3t)$ and
$-A_i-1=[(3t-1)n/(3t),n-1]-i(n/3t)$, which are disjoint for each
$i$, so that there are no monochromatic triples for that equation.
Thus the lower bound for rainbow triples given in  Corollary
\ref{cor:density} is also best possible. Moreover, the same example
for $t=1$ shows that, for $\alpha_c=0$, there are colorings of
orthogonal arrays which have no monochromatic, and hence no rainbow,
triples.

The lower bound on the number of monochromatic triples in Corollary
\ref{cor:(3,2)-3} is also best possible. Consider for example Schur
triples in a group $G$, triples of the form $(x,y,z)$ with $xy=z$.
The set of Schur triples in a finite group forms an orthogonal array
$OA(3,2)$. Alekseev and Sachev \cite{as} proved that every
equinumerous $3$-coloring of the integers in $[1,3n]$ contains a
rainbow Schur-triple. The result was later improved by Sch\"onheim
\cite{schonheim} who proved that any $3$-coloring of the integers in
$[1,N]$ such that the smallest color class has more than $N/4$
elements contains a rainbow Schur triple, and this lower bound is
best possible. The following example shows that for finite groups
there are also  $3$--colorings with no rainbow Schur triples such
that the smaller color class has cardinality $n/4$.

\begin{example}\label{exm:rainbowfree}
{\rm  Let $K<H<G$ be two subgroups of a finite group  $G$ such that
$K$ has index two in $H$ and $H$ has index two in $G$. Give color
$1$ to the elements in $K$, color $2$ to the elements in $H\setminus
K$ and color by $3$ the remaining elements of the group. In this
example $X_1X_2=X_2$ and $X_1X_3=X_2X_3=X_3$. Thus there are no
rainbow Schur triples under this coloring.}\qed
\end{example}

It is not clear to us that the lower bound $n/4$ for the size of the
smaller color class is tight in the case of three--colorings of
groups with no rainbow Schur triples.

Theorem \ref{thm:2m-r} can also be used to estimate the minimum
number of monochromatic triples in colorings of the integers.
Robertson and Zeilberger \cite{rz} showed that the minimum number of
monochromatic Schur triples in a  two coloring of the integer
interval $[1,n]$
 is $n^2/11+O(n)$.
These authors exhibit a coloring with color classes of density
$6/11$ and $5/11$ which attains the lower bound. The same result was
obtained by Schoen \cite{schoen} and Datskovsky \cite{d}. The later
author used Corollary \ref{cor:(3,2)-2} as an intermediate step of
his proof. By using Theorem \ref{thm:2m-r} one can obtain a simple
proof of a lower bound on the number of Schur triples in an {\it
equitable} coloring of $[1,n]$ with an arbitrary number of colors.

\begin{theorem}\label{thm:schur} Any equitable $r$--coloring of the integer interval $[1,n]$ has at least
$$
|M|\ge (1/2r^2)n^2+O(n)
$$
monochromatic Schur triples.
\end{theorem}

\begin{proof} Let $N=2n$ and consider the $(r+1)$ coloring $\{X_1,X_2,\ldots , X_r, X_{r+1}\}$ of the cyclic group $\Z/N\Z$ where
$\{X_1,X_2,\ldots , X_r\}$ is the given three--coloring of $[1,n]$
and $X_{r+1}=[n+1,2n]$ (we identify the integers in $[1,2n]$ with
its representatives modulo $N$). We consider the  Schur triples of
$\Z/N\Z$ ordered as $(x,y,z)$ with $x+y=z$.

By \eqref{eq:2m-r} we have
\begin{eqnarray}
2|M'|-|R'|&=&(3\sum_{i=1}^{r+1}c_{i}^2-1)(2n)^2=(3(r(1/2r)^2+(1/2)^2)-1)4n^2\nonumber\\&=&-(1-3/r)n^2,\label{eq:linealschur}
\end{eqnarray}
where $M', R'$ are the sets of monochromatic and rainbow Schur
triples respectively, and $c_{i}=|X_i|/2n$ are the densities of the
color classes.

There are $|M_{X_{r+1}}|=n^2/2+O(n)$ Schur triples of color
$X_{r+1}$. The number $\sum_{i=1}^r|M_{X_i}|$ of monochromatic Schur
triples of the other colors coincides, up to $O(n)$ terms, with the
number  $|M|$ of monochromatic triples in the given coloring of the
integer interval $[1,n]$.

Let us estimate the number $|R'|$ of rainbow triples. For each $u\in
Y$, $Y\in \{ X_1,\ldots,X_r\}$, we have the triples $$(u,w-u,w),\;
w\in [1,u]\setminus Y,$$ and the triples $$(u-w,w,u),\; w\in
[u,n]\setminus Y.$$ Therefore, for each $u\in [1,n]$ there are
$(1-1/r)n+O(1)$ such rainbow triples (each counted twice according
to the permutation of the first two coordinates) giving rise to
$$(1-1/r)n^2+O(n)$$ rainbow Schur triples with the third coordinate
in $\cup_{i=1}^rX_i$. On the other hand, for each $u\in Y$, $Y\in \{
X_1,\ldots,X_r\}$, there also the rainbow Schur triples of the form
$$(n-u,w, n+w-u) \mbox{ and  } (w,n-u,n+w-u),\; w\in [u,n]\setminus
Y$$ with the third coordinate in $X_{r+1}$. There are at least
$(r-1)n/r-u+O(1)$ choices for such  $w$, giving a total of at least
\begin{eqnarray*}
2\sum_{u=1}^{(r-1)n/r} ((r-1)n/r-u+O(1))&=&
2(1-1/r)^2n^2-(1-1/r)^2n^2+O(n)\\&=&(1-1/r)^2n^2+O(n)
\end{eqnarray*}
such rainbow triples. By plugging this estimation in
\eqref{eq:linealschur} we get
$$
|M|\ge \frac{1}{2} \left(
((1-1/r)+(1-1/r)^2-1-(1-3/r))n^2+O(n)\right)=(1/2r^2)n^2+O(n).
$$

\end{proof}

Let us consider next $3$--term arithmetic progressions. Let $G$ be a
finite group and denote by $p(G)$ the smallest prime divisor of
$|G|$. A $d$-term arithmetic progression in a  finite group $G$ with
$p(G)\ge d$ is a set of the form $\{ a, ax, ax^2,\ldots ,ax^{k-1}\}$
where $a, x\in G$. When $G$ is abelian the set $AP(3)$ of $3$--term
arithmetic progressions correspond to solutions of the equation
$x-2y+z=0$.

By proving a conjecture in \cite{jlmnr} it was shown in \cite{ms}
that a $3$--coloring of an abelian group $G$ of order $n$ such that
the smaller color class has cardinality at least $n/2p(G)$ does have
rainbow $AP(3)$, and there are  three--colorings of abelian groups
in which the smallest color class has density $1/6$ and  are free of
rainbow $AP(3)$. For a non necessarily abelian group $G$ the
following can be proved.

\begin{corollary}\label{cor:3ap} A $3$--coloring of a
group $G$ with $p(G)>53$ with smaller color class of cardinality
$\alpha n$  has at least $(6\alpha (2-3\alpha)-29/15)n^2$  rainbow
$AP(3)$. In particular, if $\alpha>(0.2725)n$ then there is a
rainbow $AP(3)$.
\end{corollary}

\begin{proof} By
Corollary \ref{cor:(3,2)-3}, the number $|R|$ of rainbow $AP(3)$
satisfies
\begin{equation}\label{eq:Ratlast}
|R|=2|M|-9\sigma^2_cn^2.
\end{equation}
We have
$$
9\sigma_c^2=3\sum_{i=1}^3c_i^2-1\le
3(2\alpha^2+(1-2\alpha)^2)-1=18\alpha^2-12\alpha+2=6\alpha
(2\alpha-3)+2.
$$
For a group $G$ with $p(G)>53$, it is shown in \cite{ccs} that every
$3$--coloring of $G$ has at least $n^2/30$ monochromatic $AP(3)$. By
substitution in \eqref{eq:Ratlast} we get
$$
|R|\ge (1/15+6\alpha (3-2\alpha)-2)n^{2}=(6\alpha
(2-3\alpha)-29/15)n^2.
$$
The last part of the statement follows since the  coefficient of
$n^2$ in the above equation is positive if  $\alpha>(0.2725)n$.
\end{proof}

The equations $x+y-z=0$ and $x-2y+z=0$ for Schur triples and
$3$--term arithmetic progressions in abelian groups are examples of
{\it regular} equations, namely, equations of the form $ax+by+cz=0$
such that the sum of a nonempty subset of the coefficients is zero.
As another consequence of Corollary \ref{cor:(3,2)-3} we have the
following result concerning rainbow solutions of such equations.

\begin{corollary}\label{cor:rainbowpr} Let $ax+by+cz=0$ be a regular equation in an abelian group $G$ of order $n$.
For every equitable $3$--coloring of  $G$ there are at least $2n$
rainbow solutions of the equation.
\end{corollary}

\begin{proof} If $a+b+c=0$ then the system has the $n$ solutions $\{ (x,x,x): \; x\in G\}$.
If $a+b=0$ then the system has the $n$ solutions $\{(x,-x,0):\; x\in G\}$.
For an equitable $3$--coloring we have $\sigma_c^2=0$. Therefore Corollary \ref{cor:(3,2)-3}
gives $|R|\ge 2|M|\ge 2n$.
\end{proof}

\section{Color patterns in  $OA(d,d-1)$}\label{sec:d,d-1}

For orthogonal arrays $OA(d,d-1)$ with arbitrary $d\ge 3$, Lemma
\ref{lem:counting} gives the following relation.

\begin{theorem}\label{thm:d,d-1} Let  $S$ be an orthogonal array
$OA(d,d-1)$ on  a set  $X$ with cardinality $n$. For each
$r$--coloring of $X$ and each color class $X_i$ we have
$$
|S_i|+(-1)^{d-1}|M_i|=((1-c_i)^d-(-1)^dc_i^d)n^{d-1},
$$
where $S_i$ denotes  the set of vectors in $S$ which miss color $i$
and $M_i$ is the set of monochromatic vectors of color $i$. In
particular, the total number $|M|$ of monochromatic vectors
satisfies
\begin{equation}\label{eq:d,d-1}
\sum_{i=1}^r|S_i|+(-1)^{d-1}|M|=\sum_{i=1}^r\left((1-c_i)^d-(-1)^dc_i^d\right)n^{d-1}.
\end{equation}
\end{theorem}

\begin{proof} Without loss of generality we may assume $i=1$. Consider the alternating sum of the equations
(\ref{eq:counting}) for vectors of  type ${\mathbf u}_j=(j,0,\ldots
,0),\; j=0,1,\ldots ,d-1$.  We have
\begin{eqnarray}
\sum_{j=0}^{d-1}(-1)^j{d\choose
j}|X_1|^j|X|^{d-1-j}&=&\sum_{j=0}^{d-1}(-1)^j\sum_{\|\vv\|=d}
{v_1\choose j}s(\vv)\nonumber
\\&=&\sum_{\|\vv\|=d}\left( \sum_{j=0}^{d-1}(-1)^j {v_1\choose
j}\right)s(\vv)\nonumber
\\&=&\sum_{\|\vv\|=d, v_1=0}s(\vv)+(-1)^{d+1}s(d,0,\ldots ,0)\nonumber
\\&=&|S_1|+(-1)^{d+1}|M_1|\label{eq:alternating},
\end{eqnarray}
where $S_1$ denotes the set of vectors which miss color $1$ and
$M_1$ denotes the set of vectors with all entries of color $1$. The
first term of the above equalities can be written as
$$\frac{1}{n}\left( (n-|X_1|)^d-(-1)^d|X_1|^d\right)=\left((1-c_1)^d-(-1)^dc_1^d\right)n^{d-1},$$
which gives the first part of the statement. Equality
\eqref{eq:d,d-1}   is  simply obtained by adding up the equations
\eqref{eq:alternating} for $i=1,\ldots ,d$.
\end{proof}

For $2$--colorings, $S_1$ and $S_2$ are just  the set of
monochromatic vectors of color $2$ and $1$ respectively. Thus
equation \eqref{eq:d,d-1} shows that, for $d$ odd, the number of
monochromatic vectors depends only on the cardinalities of the color
classes independently of their distribution, which is the main
result in \cite{ccs}.

In particular we get the following Corollary for $3$--colorings,
which is a slight generalization of a result by Balandraud
\cite[Corollary 2]{balandraud}. Here a vector is said to be {\it
rainbow} if all colors are present.

\begin{corollary}\label{d,r=3} Let  $S$ be an orthogonal array
$OA(d,d-1)$ on   $X$. For each   $3$--coloring of $X$ we have
$$
(1+(-1)^{d-1})|M|-|R|=\left(\sum_{i=1}^3(
(1-c_i)^d-(-1)^dc_i^d)-1\right)n^{d-1}.
$$
\end{corollary}

\proof With our current notion of rainbow vectors we have $$
|R|=|\cap_{j=1}^3 \bar{S_j}|=|S|-\sum_{i=1}^3|S_i|+|M|.$$ By
substitution in the last equation of Theorem \ref{thm:d,d-1} we have
\begin{eqnarray*}
\sum_{i=1}^3|S_i|+(-1)^{d+1}|M|&=&|X|^{d-1}-|R|+(1+(-1)^{d+1})|M|\\
&=&\sum_{i=1}^r\left((1-c_i)^d-(-1)^dc_i^d\right)n^{d-1},
\end{eqnarray*}
as claimed.\qed

It follows from Corollary \ref{d,r=3} that, for $d$ even, the number
of rainbow vectors in a $3$--coloring of the base set of an
orthogonal array $OA(d,d-1)$ depends only on the cardinality of the
color classes but not on the distribution of the colors.

\section{Color patterns in $OA(d,k)$}\label{sec:d,k}

Orthogonal arrays $OA(d,k)$ include sets of solutions of linear
systems: for a  $(d\times m)$ integer matrix $A$ such that every
$(m\times m)$ submatrix is nonsingular, the set of solutions of the
linear system $Ax=b$ in an abelian group forms an orthogonal array
$OA(d,d-m)$. In this general context Lemma \ref{lem:counting}  still
provides some information  on the distribution of color patterns.
Recall that a color pattern in an $r$--coloring of the base set of
an orthogonal array $OA(d,k)$ is identified by a vector
$\vv=(v_1,\ldots ,v_r)$ with $\sum_iv_i=d$ where entry $v_i$ denotes
the number of appearances of color $i$.  We consider the distance
between two color $r$--vectors  $\uu, \ww$ given by
$d(\uu,\ww)=\sum_i|u_i-w_i|$.

\begin{theorem}\label{thm:asym} Let $S$ be an orthogonal array  $OA(d,k)$ on a set  $X$ and let $c$ be
an $r$--coloring of $X$ with $\alpha=\min_ic_i$. For each color
pattern $\vv=(v_1,\ldots ,v_r)$ there is  a color pattern
$\vv'=(v'_1,\ldots ,v'_r)$   at distance  $d(\vv,\vv')\le 2(d-k)$
such that there are at least $$ s(\vv')\ge \frac{1}{{d-k+r-1\choose
r-1}}  (\alpha n)^k $$ vectors of $S$ colored with $\vv'$.
\end{theorem}

\begin{proof} We say that an $r$--vector $\ww$ dominates the $r$--vector $\uu$,
written $\ww\succeq\uu$,  if $w_1\ge u_1,\ldots ,w_r\ge u_r$.

Let $\vv$ be a given color pattern and choose $\uu =(u_1,\ldots
,u_r)$ with $|\uu|=k$ such that $\vv\succeq\uu$.  For every vector
$\vv'$ with $|\vv'|=d$ we have $ {\vv'\choose \uu}\le {d\choose
\uu}. $ By equation  \eqref{eq:counting},
$$
\left( {d\choose \uu} c_1^{u_1}\cdots
c_r^{u_r}\right)n^k=\sum_{\vv'\succeq \uu} {\vv'\choose \uu} s(\vv')
\le {d\choose \uu} \sum_{\vv'\succeq \uu} s(\vv').
$$
There are at most   ${d-k+r-1\choose r-1}$ vectors $\vv'$ with
$|\vv'|=d$  which dominate $\uu$, and each such $\vv'$ is at
distance $d(\vv',\vv)\le d(\vv',\uu)+d(\vv,\uu)\le 2(d-k)$ from
$\vv$. Hence,  if $\alpha=\min\{c_1,\ldots ,c_r\}$,
$$
\max_{\vv'\succeq \uu}s(\vv')\ge \frac{1}{ {d-k+r-1\choose r-1}}
(c_1^{u_1}\cdots c_r^{u_r})  n^k\ge \frac{1}{ {d-k+r-1\choose r-1}}
(\alpha n)^k.
$$
\end{proof}

\end{document}